%% file: main.tex
\documentclass[12 pt]{article}
\input{header.tex}

\title{Quasi Neighborhood Balanced Coloring of Graphs}

\authorentry{Maurice Genevieva Almeida}{}{p20230078@goa.bits-pilani.ac.in}{}{}

\affilentry{}{Department of Mathematics,\\ Rosary College of Commerce and Arts, Navelim, Goa, India.}

\begin{document}

\date{}
\maketitle
\begin{abstract}

\input{abstract}
 \end{abstract}


\section{Introduction}
\input{introduction}

 \section{Results} \label{sec:results}

\input{Results}

\section{Hardness Results} \label{sec:hardness}
\input{Hardness}




\bibliography{bibliography.bib}

\end{document}

%% file: header.tex
\usepackage[letterpaper,margin=1in]{geometry}
\usepackage[utf8]{inputenc}
\usepackage[english]{babel}
\usepackage{enumitem}
\usepackage{ifthen}
\usepackage{subcaption}
\usepackage[pagewise]{lineno} 
\usepackage[table]{xcolor}
\usepackage[colorlinks=true,citecolor=blue,linkcolor=blue, urlcolor=gray]{hyperref}
\usepackage{fontawesome5}
\usepackage{xparse}       
\usepackage{etoolbox}     
\usepackage{authblk}      
\usepackage{xspace}
\usepackage{multirow}
\usepackage[font=small]{caption}

\makeatletter
\let\@authorline\@empty    
\let\@affilblock\@empty    
\newcommand{\emailicon}[1]{\href{mailto:#1}{\textsuperscript{\faEnvelope[regular]}}}
\newcommand{\orcidicon}[1]{\href{https://orcid.org/#1}{\textsuperscript{\faOrcid}}}
\newcommand{\homepage}[1]{\href{#1}{\textsuperscript{\faHome}}}
\newcommand{\authorentry}[5]{%
  \ifx\@authorline\@empty
  \else
    \g@addto@macro\@authorline{,\ }%
  \fi
  \g@addto@macro\@authorline{%
    #1%
    \ifx&#2&\else\textsuperscript{#2}\fi%
    \ifx&#3&\else\emailicon{#3}\fi%
    \ifx&#4&\else\orcidicon{#4}\fi%
    \ifx&#5&\else\homepage{#5}\fi%
  }%
}
\newcommand{\affilentry}[2]{%
  \g@addto@macro\@affilblock{%
    \textsuperscript{#1}#2\\%
  }%
}
\renewcommand{\maketitle}{%
  \begin{center}
    {\LARGE \bfseries \@title \par}
    \vskip 1em
    {\normalsize
      \@authorline
    }
    \vskip 1em
    {\small \@affilblock}
    \vskip 1em
    {\small \@date \par}
  \end{center}
  \vskip 2em
}
\makeatother

\usepackage{amsmath}
\usepackage{amsfonts}
\usepackage{amssymb}
\usepackage{amsthm}
\usepackage{thmtools}
\usepackage{complexity}
\usepackage[capitalize,nameinlink]{cleveref} 
\crefname{equation}{Equation}{Equations}
\crefname{figure}{Figure}{Figures}
\newtheorem{theorem}{Theorem}[section]

\newtheorem{lemma}[theorem]{Lemma}
\newtheorem{corollary}[theorem]{Corollary}

\newtheorem{definition}[theorem]{Definition}

\usepackage{float}
\usepackage{tikz}
\usetikzlibrary{math}
\usetikzlibrary{calc}
\usetikzlibrary{positioning}
\tikzstyle{vtx}=[circle, draw, fill=black, inner sep=0pt, minimum width=5pt]
\tikzstyle{vtx-white}=[circle, draw, fill=white, inner sep=0pt, minimum width=5pt]

\newcommand{\nbc}{neighborhood balanced coloring\xspace}

\newcommand{\nbcd}{neighborhood balanced colored\xspace}
\newcommand{\qnbc}{quasi neighborhood balanced coloring\xspace}

\newcommand{\qnbcd}{quasi neighborhood balanced colored\xspace}

%% file: abstract.tex
For a simple graph $G = (V, E)$, a coloring of vertices of $G$ using two colors, say red and blue, is called a \emph{\qnbc} if, for every vertex of the graph, the number of red neighbors and the number of blue neighbors differ by at most one. In addition, there must be at least one vertex in $G$ for which this difference is exactly one.
If a graph $G$ admits such a coloring, then $G$ is said to be a \emph{quasi Neighborhood balanced colored graph}. We also define variants of such a coloring, like uniform \qnbc, positive \qnbc and negative \qnbc based on the color of the extra neighbor of every vertex of odd degree of the graph $G$. 
We present several examples of graph classes that admit the various variants of \qnbc. 
We also discuss various graph operations involving such graphs. 
Furthermore, we prove that there is no forbidden subgraph characterization for the class of \qnbc and show that the problem of determining whether a given graph has such a coloring is \NP-complete.

%% file: Introduction.tex
Let $G=(V,E)$ be a simple graph.
For any vertex $v\in V$, define an \textit{open neighborhood} of $v$ as a set $N(v):=\{u: uv\in E\}$.
The members of $N(v)$ are called the \textit{neighbors} of $v$ and the cardinality of $N(v)$ is called the \emph{degree} of $v$, denoted as $\deg(v)$.
Further, if we add the vertex $v$ to its open neighborhood $N(v)$, then we get the closed neighborhood of $v$, denoted by $N[v]$.
For graph-theoretic notation, we refer to Chartrand and Lesniak \cite{chart}.


Freyberg et al. \cite{nbc} introduced the concept `neighborhood balanced coloring'.
A \emph{neighborhood balanced coloring} of a graph $G$ is a vertex coloring of $G$ using two colors, say red and blue, such that each vertex has an equal number of neighbors of both colors.
This notion of coloring is somewhat similar to \emph{cordial labeling} \cite{cordial}.
In \cite{nbc}, the authors presented several necessary/sufficient conditions for a graph to admit neighborhood balanced coloring. They further presented several graph classes that admit neighborhood balanced coloring.

Minyard et al. \cite{3nbc} introduced the concept of `neighborhood balanced $3$-coloring' of a graph and gave a characterization of several classes of graphs that admit such a coloring.
A graph is said to admit a \emph{neighborhood balanced $3$-coloring} if the vertices of $G$ can be colored using exactly three colors such that each vertex has an equal number of neighbors of all three colors.

Almeida et al. \cite{knbc} generalized this concept by introducing the notion of neighborhood balanced $k$-coloring of graphs, for any positive integer $k \ge 2$. A coloring of the vertices of $G$ using $k$ colors such that every vertex has an equal number of vertices of each color in its neighborhood is called neighborhood balanced $k$-coloring, and the graph is called neighborhood balanced $k$-colored graph. They presented several classes of graphs that admit such a coloring and also studied neighborhood balanced $k$-colored graphs under various graph operations. They also proved that the problem of checking if a given graph $G$ admits neighborhood balanced $k$-coloring  is \NP-complete. Moreover, they also showed that the class of neighborhood balanced $k$-colored graphs is not hereditary.

Motivated by the notion of neighborhood balanced colorings of graphs, 
Collins et al.~\cite{collin-cnbc} introduced the concept of 
\emph{closed neighborhood balanced coloring}. 
In this setting, the vertices of a graph are colored with two colors, say red and blue, such that in the closed neighborhood of every vertex, the number of red vertices equals the number of blue vertices. 
Collins et al.~\cite{collin-cnbc} established several necessary/sufficient conditions for a graph to admit a closed neighborhood balanced coloring. 
They identified various classes of graphs that admit such a coloring and further investigated the behavior of closed neighborhood balanced colored 
graphs under different graph operations. Moreover, they demonstrated that the class of closed neighborhood balanced colored graphs is not hereditary. Almeida et al. \cite{kcnbc} further generalized this concept by introducing the notion of closed neighborhood balanced $k$-coloring of graphs, for any positive integer $k \ge 2$.

After examining these concepts, a natural question arises: \emph{What if a graph can be colored with two colors—red and blue—such that, for every vertex, the difference between the number of red and blue neighbors is at most one?} This question was earlier raised in \cite{nbc}.

In this article, we explore this question in detail and introduce a new vertex-coloring framework in which each vertex is assigned one of two colors, red or blue, under the condition that the absolute difference between the number of red and blue neighbors of every vertex is at most one. A formal definition of this coloring is presented below.

\begin{definition}\label{def:nnbc}
    Let $G$ be a graph. We color the vertices of $G$ using two colors, red and blue. Such a coloring is called a \emph{quasi neighborhood  balanced coloring} if, for every vertex $v$ in the graph, the number of red neighbors and the number of blue neighbors differ by at most one. In addition, there must be at least one vertex in $G$ for which this difference is exactly one.
If a graph $G$ admits such a coloring, then $G$ is said to be a \emph{quasi neighborhood  balanced colored graph}.
\end{definition}
Note that for a \qnbcd graph the vertices of even degree will have an equal number of neighbors of both colors.

Based on the color of the additional neighbor that some vertices have under a \qnbc, we introduce several variants of the quasi neighborhood  balanced coloring.

\begin{definition}\label{def:uniform_nnbc}
A quasi neighborhood  balanced coloring of a graph \(G\) is said to be a \emph{uniform quasi neighborhood  balanced coloring} if, for every vertex of odd degree, the number of red neighbors is exactly one more than the number of blue neighbors. If a graph $G$ admits uniform \qnbc then it is called uniform \qnbcd graph.  The color red is called \emph{dominating color} of the \qnbc.
\end{definition}
Note that in Definition \ref{def:uniform_nnbc}, the role of colors red and blue can be interchanged.\par 
\begin{definition}
A quasi neighborhood  balanced coloring of a graph $G$ is said to be \emph{positive} if every vertex of odd degree has exactly one more neighbor of its own color than of the opposite color.

Similarly, a quasi neighborhood  balanced coloring of graph $G$ is said to be \emph{negative} if every vertex of odd degree has exactly one more neighbor of the opposite color than of its own color.
\end{definition}
Note that a negative quasi neighborhood  balanced coloring of a graph having all vertices of odd degree is also its closed neighborhood balanced coloring.

In Section \ref{sec:results}, we present several examples of graphs that admit the various variants of \qnbc. We also study such graphs under various graph operations. In Section \ref{sec:hardness}, we prove that the decision problem: {\it Given a graph $G$, does it admit a \qnbc?} is \NP-complete. 

At the beginning, we fix the following notation, which we shall use throughout the article.
If $c$ is a \qnbc of $G$, then denote by $V_G^c(R)$ and $V_G^c(B)$  the number of vertices of $G$ colored red and blue respectively by the coloring $c$. When the coloring or the graph under consideration is clear from the context, we shall omit the subscript or the superscript from the above notations.

%% file: Results.tex
The following lemma follows trivially from Definition \ref{def:nnbc}.
\begin{lemma}\label{lem:deg_nnbc}
    A \qnbc graph has at least one vertex of odd degree. 
\end{lemma}
\subsection{Counting Results}

Let $G$ be a graph that admits a quasi neighborhood  balanced coloring  $c$.  
Let $V(R)$ and $V(B)$ denote the sets of vertices of $G$ colored red and blue, respectively by $c$.  
Each of these sets is further partitioned according to the neighborhood imbalance:

\begin{itemize}
    \item $V(R_1)$: red vertices having more red than blue neighbors;
    \item $V(R_2)$: red vertices having more blue than red neighbors;
    \item $V(B_1)$: blue vertices having more blue than red neighbors;
    \item $V(B_2)$: blue vertices having more red than blue neighbors.
\end{itemize}

Consider the bipartite subgraph $H$ induced by the edges between $V(R)$ and $V(B)$.
The number of edges of $H$ satisfies
\[
|E(H)| = \sum_{v\in V(R)} \deg_H(v)
       = \sum_{v\in V(B)} \deg_H(v).
\]

Using the \qnbc\ condition, we obtain the following expression for $|E(H)|$:
\begin{align}\label{eq:counting1}
    |E(H)|
    &= \sum_{v\in V(R_1)} \frac{\deg_G(v)-1}{2}
     + \sum_{v\in V(R_2)} \frac{\deg_G(v)+1}{2} \\[2mm]
    &= \sum_{v\in V(B_1)} \frac{\deg_G(v)-1}{2}
     + \sum_{v\in V(B_2)} \frac{\deg_G(v)+1}{2}. \notag
\end{align}

Rearranging~\eqref{eq:counting1}, we obtain
\begin{align}\label{eq:counting2}
    \sum_{v\in V(R)} \deg_G(v)
    + |V(R_2)| - |V(R_1)|
    =
    \sum_{v\in V(B)} \deg_G(v)
    + |V(B_2)| - |V(B_1)|.
\end{align}

Since
\[
2|E(G)| = \sum_{v\in V(R)}\deg_G(v) + \sum_{v\in V(B)}\deg_G(v),
\]
substituting \eqref{eq:counting2} yields
\begin{align*}
    2|E(G)|
    &= 2\sum_{v\in V(B)}\deg_G(v)
     + |V(B_2)| - |V(B_1)|
     + |V(R_1)| - |V(R_2)|.
\end{align*}

Equation~\eqref{eq:counting1} also gives
\[
2|E(H)| - |V(B_2)| + |V(B_1)| = \sum_{v\in V(B)}\deg_G(v).
\]

Substituting this expression into the previous equality, we obtain
\begin{align*}
    2|E(G)|
    &= 4|E(H)|
     - |V(B_2)| + |V(B_1)|
     - |V(R_2)| + |V(R_1)| \\
    &= 4|E(H)| - |V(B_2)\cup V(R_2)| + |V(B_1)\cup V(R_1)|.
\end{align*}
Thus, $$|E(H)|=\frac{2|E(G)|-r+s}{4},$$ where $r=|V(R_2)\cup V(B_2)|$ and $s=|V(R_1)\cup V(B_1)|$.\\\\
Observe that:
\begin{itemize}
    \item For a \emph{positive} \qnbc, $r=0$ and $s$ equals the number of odd-degree vertices.
    \item For a \emph{negative} \qnbc, $s=0$ and $r$ equals the number of odd-degree vertices.
\end{itemize}

These calculations yield the following results.

\begin{theorem}
Let $G$ be a graph admitting a quasi neighborhood  balanced coloring.
\begin{enumerate}
    \item If $G$ admits a positive \qnbc, then
    \[
        2|E(G)| \equiv -k\ (\bmod\ {4}),
    \]
    \item If $G$ admits a negative \qnbc, then
    \[
        2|E(G)| \equiv k\ (\bmod\ {4}),
    \]
\end{enumerate}
 where $k$ denotes the number of vertices of odd degree.
\end{theorem}
\begin{corollary}
    For a (odd) regular graph $G$ admitting a \qnbc,
   \begin{enumerate}
       \item If $G$ admits a positive \qnbc, then
    \[
        2|E(G)| \equiv -|V(G)|\ (\bmod\ {4}),
    \]
    \item If $G$ admits a negative \qnbc, then
    \[
        2|E(G)| \equiv |V(G)|\ (\bmod\ {4}).
    \]
   \end{enumerate}
\end{corollary}


\subsection{Some quasi neighborhood  Balanced Colored Graphs}
\begin{theorem}
    The complete graph $K_n$ is negative \qnbcd if and only if $n$ is even.
\end{theorem}
\begin{proof}
When $n$ is even, color exactly $\frac{n}{2}$ vertices of $K_n$ red and the remaining $\frac{n}{2}$ vertices blue. Then every vertex has $\frac{n}{2}-1$ neighbors of its own color and $\frac{n}{2}$ neighbors of the opposite color, so the difference between the two counts is exactly $1$. Hence this coloring is a negative quasi neighborhood  balanced coloring of $K_n$.

For $n$ odd, every vertex of $K_n$ has even degree. By Lemma~\ref{lem:deg_nnbc}, no such graph admits a quasi neighborhood  balanced coloring. Therefore $K_n$ is not a \qnbcd graph when $n$ is odd.
\end{proof}
Let $1\leq d\leq \frac{n-1}{2}$. The generalized Petersen graph $GP(n,d)$ consists of a cycle $C_n$ with consecutive vertices $v_0,v_1,\dots,v_{n-1}$ and additional vertices $u_0,u_1,\dots,u_{n-1}$ with an edge between $v_i$ and $u_i$ for $0\leq i \leq n-1$, and edges between $u_i$ and $u_{i+d}$ for $0\leq i \leq n-1$, with subscript arithmetic modulo $n$. 
\begin{theorem}
    The graph $GP(n,d)$ admits positive \qnbc for all positive integers $n$ and $d$ with $1\leq d \leq\frac{n-1}{2}$.
\end{theorem}
\begin{proof}
Color the vertices $v_0, v_1, \dots, v_{n-1}$ red and the vertices $u_0, u_1, \dots, u_{n-1}$ blue.  
Under this coloring, each vertex $v_i$ ($0 \le i \le n-1$), colored red, has two red neighbors and one blue neighbor, while each vertex $u_i$ ($0 \le i \le n-1$), colored blue, has two blue neighbors and one red neighbor.  
Thus, every vertex has more neighbors of its own color, and hence the coloring is a positive \qnbc.
\end{proof}
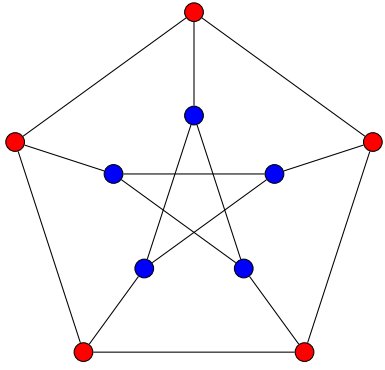
\begin{figure}
    \centering
    \input{Petersen}
    \caption{Positive \qnbc of $GP(5,1)$.}
    \label{fig:placeholder}
\end{figure}
\begin{theorem}
    The complete bipartite graph $K_{m,n}$ admits an uniform \qnbc if and only if at least one of $m$ or $n$ is odd.
\end{theorem}
\begin{proof}
If both $m$ and $n$ are even, then every vertex of $K_{m,n}$ has even degree. Hence, by Lemma~\ref{lem:deg_nnbc}, the graph $K_{m,n}$ cannot admit a quasi neighborhood  balanced coloring.

Conversely, suppose that at least one of $m$ or $n$ is odd. Assume first that $m$ is odd. In the partite set of size $m$, color $\tfrac{m-1}{2}$ vertices red and $\tfrac{m+1}{2}$ vertices blue. Similarly, if $n$ is odd, color $\tfrac{n-1}{2}$ vertices red and $\tfrac{n+1}{2}$ vertices blue in the partite set of size $n$. 

Under this coloring, every vertex of odd degree has exactly one more blue neighbor than red neighbors, while every vertex of even degree has the same number of neighbors of both colors. Therefore $K_{m,n}$ is a uniform \qnbcd graph whenever at least one of $m$ or $n$ is odd.
\end{proof}
\begin{corollary}
    The star $K_{1,n}$ admits uniform \qnbc for all $n$.
\end{corollary}

\begin{theorem}
The path $P_n$ admits \qnbc for all $n$. In addition, the path $P_n$ admits uniform \qnbc if and only if $n\not\equiv1\pmod{4}$.    
\end{theorem}
 \begin{proof}
Let $V(P_n)=\{v_1,v_2,\dots,v_n\}$ and $E(P_n)=\{v_iv_{i+1};\;1\le i\le n-2\}$ denote the vertex set and edge set of $P_n$, respectively.

\begin{description}

\item[Case 1:] $n\equiv 2,3 \pmod{4}$.\\
Define a coloring $c\colon V(P_n)\to\{\text{red},\text{blue}\}$ by
\[
c(v_i)=
\begin{cases}
\text{red}, & \text{if}\ i\equiv 1,2 \pmod{4},\\[2pt]
\text{blue}, & \text{if}\ i\equiv 3,0 \pmod{4}.
\end{cases}
\]
By construction, vertices at distance two receive different colors.  
Each non-pendant vertex $v_i$ ($2\le i\le n-1$) has neighbors $v_{i-1}$ and $v_{i+1}$, which under this coloring always receive opposite colors; hence every such vertex has one red neighbor and one blue neighbor.  
The pendant vertices $v_1$ and $v_n$ have neighbors $v_2$ and $v_{n-1}$, respectively. Since $n\equiv 2,3\pmod{4}$, both $v_2$ and $v_{n-1}$ receive the same color (red).  
Thus, the coloring $c$ is a uniform \qnbc{} of $P_n$.

\item[Case 2:] $n\equiv 0 \pmod{4}$.\\  
Define a coloring $c\colon V(P_n)\to\{\text{red},\text{blue}\}$ by
\[
c(v_i)=
\begin{cases}
\text{red}, & \text{if}\ i\equiv 0,1 \pmod{4},\\[2pt]
\text{blue}, & \text{if}\ i\equiv 2,3 \pmod{4}.
\end{cases}
\]
As in Case~1, vertices at distance two receive different colors, and every non-pendant vertex $v_i$ has neighbors $v_{i-1}$ and $v_{i+1}$ of opposite colors.  
For the pendant vertices $v_1$ and $v_n$, their neighbors $v_2$ and $v_{n-1}$ respectively, both receive blue when $n\equiv 0\pmod{4}$.  
Hence, the coloring $c$ is a uniform \qnbc{} of $P_n$.

\item[Case 3:] $n\equiv 1 \pmod{4}$.\\  
Without loss of generality, color $v_1$ red. Two subcases arise depending on the color of $v_2$.

\smallskip
\emph{Subcase A: $c(v_2)=\text{red}$.} \\ 
Then $v_3$ and $v_4$ must be blue, and the repeating pattern \(\text{red},\text{red},\text{blue},\text{blue}\) continues along the path.  
In this pattern, $v_{n-1}$ receives blue while $v_2$ receives red, implying that the two pendant vertices have neighbors of different colors.

\smallskip
\emph{Subcase B: $c(v_2)=\text{blue}$.} \\ 
Then $v_3$ must be blue and $v_4$ red, giving the repeating pattern \(\text{red},\text{blue},\text{blue},\text{red}\).  
In this pattern, $v_{n-1}$ receives red while $v_2$ receives blue, again resulting in pendant neighbors that are differently colored.

\smallskip
In both subcases, the pendant vertices have neighbors of different colors, and hence $P_n$ does not admit a uniform \qnbc when $n\equiv 1\pmod{4}$.  
However, each of the colorings constructed above are still \qnbc, so $P_n$ does admit a \qnbc in this case too.

\end{description}
\end{proof}
\begin{figure}
    \centering
    \input{Paths}
    \caption{Quasi neighborhood balanced coloring of Paths $P_5,P_6,P_7,$ and $P_8$.}
    \label{fig:placeholder}
\end{figure}
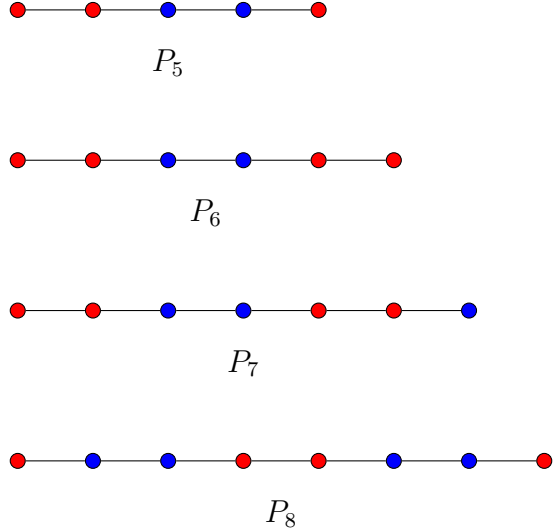

The bistar $B_{m,n}$ is a tree with vertex set $V(B_{m,n})=\{a,b\}\cup \{u_i: 1\leq i \leq m\}\cup\{v_i: 1\leq i \leq n\}$ and edge set $E(B_{m,n})=\{ab\}\cup \{au_i: 1\leq i \leq m\}\cup \{bv_i: 1\leq i \leq n\}$. We now give conditions on $m,n$ for the bistar $B_{m,n}$ to admit a \qnbc.
\begin{theorem}
 The bistar $B_{m,n}$ admits uniform \qnbc for all positive integers $m$ and $n$.   
\end{theorem}
\begin{proof}
For $n$ even, color ${n}/{2}$ pendant vertices red and the remaining ${n}/{2}$ blue.  
For $n$ odd, color ${(n-1)}/{2}$ pendant vertices red and the remaining ${(n+1)}/{2}$ blue. Use a similar approach to color the other $m$ pendant vertices depending upon the parity of $m$.  
In addition, color the two non-pendant vertices $a$ and $b$ red.

In all cases, each pendant vertex has its unique neighbor colored red.  
Now consider the non-pendant vertices $a$ and $b$.  

If $n$ (or $m$) is even, then each of $a$ and $b$ has exactly one more red neighbor than blue neighbors, since the pendant vertices contribute an equal number of red and blue neighbors.  

If $n$ (or $m$) is odd, then each of $a$ and $b$ has an equal number of red and blue pendant neighbors.  

Thus, in every case, the difference between the number of red and blue neighbors of every vertex is at most one, with any excess always in favor of red. Hence, the resulting colouring is a uniform \qnbc.
\end{proof}
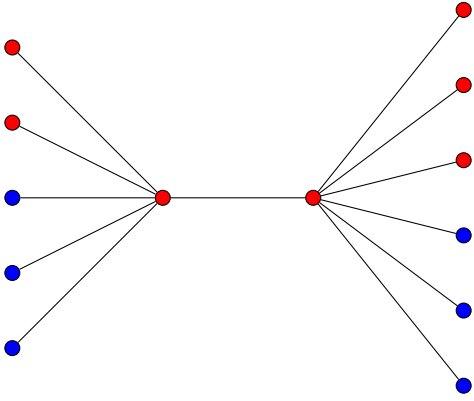
\begin{figure}
    \centering
    \input{Bistar}
    \caption{Uniform \qnbc of Bistar $B_{5,6}$.}
    \label{fig:placeholder}
\end{figure}


\subsection{\small Operations with quasi neighborhood  Balanced Colored Graphs}
For any of the graph products and a fixed vertex $u$ of $G$, the set of vertices $\{(u,v):\ v\in V(H)\}$ is called an $H$-layer, and we denote it by $H_u$. Similarly, if $v\in V(H)$ is fixed, then the set of vertices $\{(u,v):\ u\in V(G)\}$ is called a $G$-layer and we denote it by $G_v$.
If one constructs $V(G)\times V(H)$ in a natural way, the $H$-layers are represented by rows and the $G$-layers are represented by columns.\par \bigskip
The {\it lexicographic product} of $G$ and $H$, denoted by $G[H]$, is a graph with vertex set $V(G)\times V(H)$ and two vertices $(u,v)$ and $(u',v')$ are adjacent if and only if either $uu'\in E(G)$ or $u=u'$ and $vv'\in E(H)$.
It may be instructive to instead construct $G[H]$ by replacing every vertex of $G$ with a copy of $H$ and then replacing each edge of $G$ with a complete bipartite graph between the corresponding copies of $H$.
\begin{theorem}
    Let $G$ and $H$ be two graphs. If $H$ admits \qnbc $c$ with $|V^c_H(R)|=|V^c_H(B)|$, then $G[H]$ admits \qnbc. 
\end{theorem}
   \begin{proof} Color vertices of $G[H]$ as follows.
Apply the quasi neighborhood  balanced coloring $c$ of $H$ to every $H$-layer in the lexicographic product $G[H]$. We claim that this induces a quasi neighborhood  balanced coloring of $G[H]$.

Consider any vertex $(u,v) \in G[H]$. Since $c$ is a \qnbc of $H$, the difference between the number of red neighbors and blue neighbors of $(u,v)$ within its own layer $H_u$ is at most one.

Now look at the neighbors of $(u,v)$ outside the $H$-layer $H_u$. For each neighbor $u'$ of $u$ in $G$, the vertex $(u,v)$ is adjacent to every vertex in the layer $H_{u'}$. Hence, outside its layer, $(u,v)$ has
$\deg_G(u)\cdot |V_H^c(R)|$ red neighbors and 
$\deg_G(u)\cdot |V_H^c(B)|$ blue neighbors.
Further, as $|V_H^c(R)| = |V_H^c(B)|$, the contributions of red and blue neighbors from all other layers is equal.

Combining these two observations, the total difference between the number of red and blue neighbors of $(u,v)$ in $G[H]$ is at most one. Thus the coloring obtained is a quasi neighborhood  balanced coloring of $G[H]$, as claimed.
\end{proof}
Using the same approach as in the above theorem, we can show that if the graph
\(H\) admits a uniform/positive/negative \qnbc with
\(|V_H^c(R)| = |V_H^c(B)|\), then the lexicographic product \(G[H]\)
also admits a uniform/positive/negative \qnbc.

\begin{theorem}
    If $G$ is \nbcd graph and $H$ is \qnbcd graph, then $G[H]$ is \qnbcd graph.  
\end{theorem}
\begin{proof}
Let $g$ be a neighborhood balanced coloring of $G$ and let $h$ be a quasi neighborhood  balanced coloring of $H$. We define a coloring of the lexicographic product $G[H]$ as follows.

Fix a vertex $u \in V(G)$ and consider any vertex $(u,v) \in V(G[H])$. If $g(u)$ is red, then we color the entire $H$-layer $H_u$ using the coloring $h$. If $g(u)$ is blue, then we color the entire $H$-layer $H_u$ using the complementary coloring $\bar{h}$. In this way, all the vertices of $G[H]$ are colored. We claim that this is a \qnbc of $G[H]$.

Consider an arbitrary vertex $(u,v)\in V(G[H])$. Since $G$ is a \nbc graph, we may assume that $u$ has exactly $k$ red neighbors and $k$ blue neighbors in $G$ for some integer $k$. Because $h$ (and therefore $\bar{h}$) is a \qnbc of $H$, the difference between the number of red and blue neighbors of $(u,v)$ within its own layer $H_u$ is at most one.

It remains to analyze the neighbors of $(u,v)$ lying outside $H_u$. Each red neighbor of $u$ in $G$ contributes a full copy of $H$, colored according to $h$, and each blue neighbor contributes a full copy of $H$, colored according to $\bar{h}$. Hence the number of red neighbors of $(u,v)$ outside $H_u$ is
$k\bigl(|V_H^h(R)| + |V_H^{\bar{h}}(R)|\bigr)$,
and the number of blue neighbors outside $H_u$ is
$k\bigl(|V_H^h(B)| + |V_H^{\bar{h}}(B)|\bigr)$.

Since $h$ and $\bar{h}$ are complementary colorings, we have
\[
|V_H^h(R)| = |V_H^{\bar{h}}(B)| \qquad\text{and}\qquad
|V_H^h(B)| = |V_H^{\bar{h}}(R)|.
\]
Therefore,
\[
|V_H^h(R)| + |V_H^{\bar{h}}(R)|
= |V_H^h(B)| + |V_H^{\bar{h}}(B)|,
\]
and thus the red and blue contributions from all other layers are equal.

Combining this with the fact that the difference within $H_u$ is at most one, it follows that $(u,v)$ satisfies the quasi neighborhood  balanced condition in $G[H]$.

Hence the defined coloring is a \qnbc of $G[H]$, completing the proof.
\end{proof}
Using the same approach as in the above theorem, we can show that if the graph
\(H\) admits a uniform/positive/negative \qnbc and $G$ admits \nbc, then the lexicographic product \(G[H]\)
also admits a uniform/positive/negative \qnbc.

The {\it direct product} of $G$ and $H$, denoted by $G\times H$, is a graph with vertex set $V(G)\times V(H)$ and two vertices $(u,v)$ and $(u',v')$ are adjacent if and only if $uu'\in E(G)$ and $vv'\in E(H)$.
\begin{theorem}\label{th:nbc_direct}
    If $G$ is \qnbcd graph and $H$ is \nbcd graph, then $G\times H$ is \nbcd graph.
\end{theorem}
\begin{proof}
Let $G$ admit a quasi neighborhood  balanced coloring $g$, and let $H$ admit a neighborhood balanced coloring $h$. We define a coloring on the direct product $G \times H$ as follows.

Fix a vertex $v \in V(H)$ and consider any vertex $(u,v) \in V(G \times H)$. If $v$ is colored red under $h$, then we color the entire $G$-layer $G_v$ using $g$. If $v$ is colored blue, we color the layer $G_v$ using the complementary coloring $\overline{g}$. In this way, every vertex of $G \times H$ receives a color.

We now show that this coloring is a neighborhood balanced coloring of $G \times H$.

Let $(u,v)$ be any vertex in $G \times H$. Suppose that, under the coloring $g$ of $G$, the vertex $u$ has exactly one more red neighbor than blue. Then in the product $G \times H$, the vertex $(u,v)$ will have one more red neighbor than blue in precisely those layers $G_{v'}$ where $v'$ is a neighbor of $v$ in $H$ that is colored red (under $h$). Likewise, $(u,v)$ will have one more blue neighbor than red in the layers corresponding to neighbors $v'$ of $v$ in $H$ that are colored blue.

Since $H$ is a neighborhood balanced graph, the vertex $v$ has an equal number of red and blue neighbors under the coloring $h$. Consequently, every vertex $(u,v)$ in $G \times H$ has exactly the same number of red and blue neighbors and thus the defined coloring is a neighborhood balanced coloring of $G \times H$, completing the proof.
\end{proof}
\begin{theorem}\label{th:nnbc_direct}
    If $G$ and $H$ are \qnbcd graphs, then so is $G\times H$.
\end{theorem}
\begin{proof}
Let $G$ and $H$ admit a quasi neighborhood  balanced colorings $g$ and $h$ respectively. We define a coloring on the direct product $G \times H$ as follows.

Fix a vertex $v \in V(H)$ and consider any vertex $(u,v) \in V(G \times H)$. If $v$ is colored red under $h$, then we color the entire $G$-layer $G_v$ using $g$. If $v$ is colored blue, we color the layer $G_v$ using the complementary coloring $\overline{g}$. In this way, every vertex of $G \times H$ receives a color.

We now show that this coloring is a \qnbc of $G \times H$.

Let $(u,v)$ be any vertex in $G \times H$. Suppose that, under the coloring $g$ of $G$, the vertex $u$ has exactly one more red neighbor than blue. Then in the product $G \times H$, the vertex $(u,v)$ will have one more red neighbor than blue in precisely those layers $G_{v'}$ where $v'$ is a neighbor of $v$ in $H$ that is colored red (under $h$). Likewise, $(u,v)$ will have one more blue neighbor than red in the layers corresponding to neighbors $v'$ of $v$ in $H$ that are colored blue.

Since $H$ is a \qnbcd graph, for the vertex $v$ the difference between the number of red and blue neighbors is at most one under the coloring $h$. Consequently, for every vertex $(u,v)$ in $G \times H$ the difference between the number of red and blue neighbors is at most one and thus the defined coloring is a \qnbc of $G \times H$, completing the proof.
\end{proof}
Using the same approach as in the above theorem, we can show that if the graphs $G$ and $H$ are uniform/positive/negative \qnbc, then the graph $G \times H$ is also uniform/positive/negative \qnbc respectively.

The {\it cartesian product}  of $G$ and $H$, denoted by $G\, \square \, H$, is a graph with vertex set $V(G)\times V(H)$ and two vertices $(u,v)$ and $(u',v')$ are adjacent if and only if $u=u'$ and $vv'\in E(H)$  or $v=v'$ and $uu'\in E(G)$.
\begin{theorem}\label{th:nbc_cartesian}
    If $G$ is \qnbcd graph and $H$ is \nbcd graph, then $G\, \square\, H$ is a \qnbcd graph. 
\end{theorem}
\begin{proof}
Let $G$ admit a quasi neighborhood  balanced coloring $g$, and let $H$ admit a neighborhood balanced coloring $h$. We define a coloring on the Cartesian product $G \,\square\, H$ as follows.

Fix a vertex $v \in V(H)$ and consider any vertex $(u,v) \in V(G \,\square\, H)$. If $v$ is colored red under $h$, we color the entire $G$-layer $G_v$ using $g$. If $v$ is colored blue under $h$, we color the layer $G_v$ using the complementary coloring $\overline{g}$. In this way, each $G$-layer is colored using either $g$ or $\overline{g}$, and likewise each $H$-layer is colored using either $h$ or its complement $\overline{h}$.

Since each $G$-layer is colored by $g$ or $\overline{g}$, the difference between the number of red and blue neighbors of $(u,v)$ of the form $(u',v)$ is at most one. Similarly, because every $H$-layer is colored using $h$ or $\overline{h}$, the vertex $(u,v)$ has an equal number of red and blue neighbors of the form $(u,v')$.

Combining these two observations, we conclude that the total difference between the number of red and blue neighbors of $(u,v)$ in $G \,\square\, H$ is at most one. Therefore, the defined coloring is a quasi neighborhood  balanced coloring of $G \, \square\, H$, as required.
\end{proof}
\begin{theorem}
    If $G$ is a positive (negative) \qnbcd graph and $H$ is \nbcd graph, then $G\,\square\, H$ is a positive (negative) \qnbcd graph 
\end{theorem}
\begin{proof}
Let \(G\) admit a positive quasi neighborhood  balanced coloring \(g\), and let
\(H\) admit a neighborhood balanced coloring \(h\).
We color the vertices of \(G \,\square\, H\) in the same manner as described in
Theorem~\ref{th:nbc_cartesian}.

Since \(g\) is a positive \qnbc of \(G\), each vertex of \(G\) has exactly one
more neighbor of its own color than of the opposite color. Moreover, as each
\(G\)-layer is colored using either \(g\) or \(\overline{g}\), any vertex
\((u,v)\) has one more neighbor of the form \((u',v)\) in its own color than in
the other color. Similarly, because each \(H\)-layer is colored using either
\(h\) or \(\overline{h}\), the vertex \((u,v)\) has an equal number of red and
blue neighbors of the form \((u,v')\).

Combining these observations, we conclude that every vertex \((u,v)\) has more
neighbors of its own color than of the opposite color. Hence, the induced
coloring is a positive \qnbc of \(G \,\square\, H\).
\end{proof}
The {\it strong product} of $G$ and $H$, denoted by $G \boxtimes H$, is a graph with vertex set $V(G)\times V(H)$ and two vertices $(u,v)$ and $(u',v')$ are adjacent if and only if $u=u'$ and $vv'\in E(H)$ or $v=v'$ and $uu'\in E(G)$ or $uu'\in E(G)$ and $vv'\in E(H)$. Note that from the definition, it is clear that the strong product is the edge-disjoint union of the cartesian and the direct product, and therefore, we have the following result.
\begin{theorem}\label{th:nnbc_join}
    If $G$ is \qnbcd graph and $H$ is \nbcd graph, then $G\,\boxtimes\, H$ is a \qnbcd graph. 
\end{theorem}
\begin{proof}
    The result follows from Theorem \ref{th:nbc_direct} and Theorem \ref{th:nbc_cartesian}.
\end{proof}
It is easy to show that if the graph $G$ is uniform/positive/negative \qnbcd graph and $H$ is \nbcd graph, then $G\,\boxtimes\, H$ is a uniform/positive/negative \qnbcd graph.\\

The {\it join of graphs} of $G$ and $H$, denoted by $G\vee H$, is a graph having vertex set $V(G)\,\cup\, V(H)$ and edge set $E(G)\cup E(H)\cup \{xy:  x\in V(G) \text{ and } y \in V(H)\}$.
\begin{theorem}
    If $g$ is a \qnbc of graph $G$   with $|V^g_G(R)|=|V^g_G(B)|$ and $h$ is a \qnbc of graph $H$ with $|V^h_H(R)|=|V^h_H(B)|$, then $G \vee H$ is \qnbcd graph.
\end{theorem}
\begin{proof}
In the join $G \vee H$, color the vertices of $G$ according to $g$ and the vertices of $H$ according to $h$.

Since $g$ is a \qnbc of $G$, for every vertex of $G$ the difference between the number of red and blue neighbors is at most one when considering only edges within $G$. Moreover, because $|V^h_H(R)| = |V^h_H(B)|$, each vertex of $G$ receives an equal number of red and blue neighbors from $H$. Thus, the difference between the number of red and blue neighbors for each vertex of $G$ remains at most one in $G \vee H$.

Similarly, since $|V^g_G(R)| = |V^g_G(B)|$, every vertex of $H$ receives an equal number of red and blue neighbors from $G$. Combined with the fact that $h$ is a \qnbc of $H$, the difference between the number of red and blue neighbors for each vertex of $H$ is also at most one in $G \vee H$.

Hence $G \vee H$ is a \qnbcd graph.
\end{proof}
Using the same approach as in the above theorem, we can show that if the graphs $G$ and $H$ admit a positive/negative \qnbc then $G\vee H$ also admits positive/negative \qnbc. Further, we can also show that, if the graphs $G$ and $H$ admit uniform \qnbc with the same dominating color, then $G\vee H$ also admits uniform \qnbc.

The following results on $G\vee H$ follow easily from arguments as presented in Theorem \ref{th:nnbc_join}. We therefore omit their proofs.
\begin{theorem}
    If $g$ is a \nbc of graph $G$ with $|V^g_G(R)|=|V^g_G(B)|+1$ and $h$ is a \nbc of graph $H$ with $|V^h_H(R)|=|V^h_H(B)|$, then $G\vee H$ is an uniform \qnbcd graph.
\end{theorem}
\begin{theorem}
    If $g$ is a \qnbc of graph $G$ with $|V_G^g(R)|=|V_G^g(B)|+1$ and $h$ is a \nbc of graph $H$ with $|V_H^h(R)|=|V_H^h(B)|$, then $G\vee H$ is a \qnbcd graph.   
\end{theorem}
\subsection{\small Non-hereditary property of the class of \qnbcd graphs}

A family $\mathcal{F}$ of graphs is hereditary if each induced subgraph of every graph in $\mathcal{F}$ is also in $\mathcal{F}$.
Further, it is known that the hereditary classes can be characterized by providing a list of forbidden induced subgraphs.
Next, we show that the class of \qnbcd graphs is not hereditary.

\begin{theorem}\label{th:heredity}
    The class of \qnbcd graphs is not hereditary. 
\end{theorem}
\begin{proof}
To prove the theorem, it is enough to show that every graph is an induced subgraph of a \qnbcd graph.

Let $G$ be a graph with $V(G)=\{v_1,v_2,\dots,v_n\}$. Construct a graph $H$ from $G$ with the vertex set $V(H)=\cup_{j=1}^3\{v_i^j;\ 1\leq i \leq n\}$ and the edge set $$E(G)=\big(\cup_{p=1}^2\{v_i^pv_j^p;\ v_iv_j\in E(G)\}\big)\cup \{v_i^1v_j^2;\ v_iv_j\in E(G) \}\cup \{v_i^2v_i^3;\ 1\leq i \leq n \}.$$ 
It is easy to see that $G$ is an induced subgraph of $H$. In $H$ for $1\leq i \leq n$, color the vertices $v_i^1$  using color red, the vertices $v_i^2$ using color blue and the vertices $v_i^3$ again using color blue. 

Note that the vertices $v_i^1$ have an equal number of neighbors of both colors. The vertices $v_i^2$ have one more blue neighbor compared to red, while the vertices $v_i^3$ also have one more blue neighbor compared to red. Hence, the graph $H$ is a \qnbcd graph.
\end{proof}

%% file: Petersen.tex
\begin{tikzpicture}[scale=2.5,
    vertex/.style={circle, draw, inner sep=2.5 pt}
]

\node[vertex,fill=red] (v1) at (90:1)   {};
\node[vertex,fill=red] (v2) at (162:1)  {};
\node[vertex,fill=red] (v3) at (234:1)  {};
\node[vertex,fill=red] (v4) at (306:1)  {};
\node[vertex,fill=red] (v5) at (18:1)   {};

\node[vertex, fill=blue] (u1) at (90:0.45)   {};
\node[vertex, fill=blue] (u2) at (162:0.45)  {};
\node[vertex, fill=blue] (u3) at (234:0.45)  {};
\node[vertex, fill=blue] (u4) at (306:0.45)  {};
\node[vertex, fill=blue] (u5) at (18:0.45)   {};

\draw (v1)--(v2);
\draw (v2)--(v3);
\draw (v3)--(v4);
\draw (v4)--(v5);
\draw (v5)--(v1);

\draw (u1)--(u3);
\draw (u3)--(u5);
\draw (u5)--(u2);
\draw (u2)--(u4);
\draw (u4)--(u1);

\draw (v1)--(u1);
\draw (v2)--(u2);
\draw (v3)--(u3);
\draw (v4)--(u4);
\draw (v5)--(u5);

\end{tikzpicture}

%% file: Paths.tex
\begin{tikzpicture}[
    scale=1,
    vertex/.style={circle, draw, inner sep=2pt}
]

\node[vertex, fill=red] (p51) at (0,0) {};
\node[vertex, fill=red] (p52) at (1,0) {};
\node[vertex, fill=blue] (p53) at (2,0) {};
\node[vertex, fill=blue] (p54) at (3,0) {};
\node[vertex, fill=red] (p55) at (4,0) {};

\draw (p51)--(p52)--(p53)--(p54)--(p55);

\node at (2,-0.7) {$P_5$};

\node[vertex, fill=red] (p61) at (0,-2) {};
\node[vertex, fill=red] (p62) at (1,-2) {};
\node[vertex, fill=blue] (p63) at (2,-2) {};
\node[vertex, fill=blue] (p64) at (3,-2) {};
\node[vertex, fill=red] (p65) at (4,-2) {};
\node[vertex, fill=red] (p66) at (5,-2) {};

\draw (p61)--(p62)--(p63)--(p64)--(p65)--(p66);

\node at (2.5,-2.7) {$P_6$};

\node[vertex, fill=red] (p71) at (0,-4) {};
\node[vertex, fill=red] (p72) at (1,-4) {};
\node[vertex, fill=blue] (p73) at (2,-4) {};
\node[vertex, fill=blue] (p74) at (3,-4) {};
\node[vertex, fill=red] (p75) at (4,-4) {};
\node[vertex, fill=red] (p76) at (5,-4) {};
\node[vertex, fill=blue] (p77) at (6,-4) {};

\draw (p71)--(p72)--(p73)--(p74)--(p75)--(p76)--(p77);

\node at (3,-4.7) {$P_7$};

\node[vertex, fill=red] (p81) at (0,-6) {};
\node[vertex, fill=blue] (p82) at (1,-6) {};
\node[vertex, fill=blue] (p83) at (2,-6) {};
\node[vertex, fill=red] (p84) at (3,-6) {};
\node[vertex, fill=red] (p85) at (4,-6) {};
\node[vertex, fill=blue] (p86) at (5,-6) {};
\node[vertex, fill=blue] (p87) at (6,-6) {};
\node[vertex, fill=red] (p88) at (7,-6) {};

\draw (p81)--(p82)--(p83)--(p84)--(p85)--(p86)--(p87)--(p88);

\node at (3.5,-6.7) {$P_8$};

\end{tikzpicture}

%% file: Bistar.tex
\begin{tikzpicture}[
    scale=2,
    vertex/.style={circle, draw, inner sep=2pt}
]

\node[vertex, fill=red] (c1) at (0,0) {};
\node[vertex, fill=red] (c2) at (1,0) {};

\draw (c1)--(c2);

\node[vertex, fill=red] (u1) at (-1,1) {};
\node[vertex, fill=red] (u2) at (-1,0.5) {};
\node[vertex, fill=blue] (u3) at (-1,0) {};
\node[vertex, fill=blue] (u4) at (-1,-0.5) {};
\node[vertex, fill=blue] (u5) at (-1,-1) {};

\draw (c1)--(u1);
\draw (c1)--(u2);
\draw (c1)--(u3);
\draw (c1)--(u4);
\draw (c1)--(u5);

\node[vertex, fill=red] (v1) at (2,1.25) {};
\node[vertex, fill=red] (v2) at (2,0.75) {};
\node[vertex, fill=red] (v3) at (2,0.25) {};
\node[vertex, fill=blue] (v4) at (2,-0.25) {};
\node[vertex, fill=blue] (v5) at (2,-0.75) {};
\node[vertex, fill=blue] (v6) at (2,-1.25) {};

\draw (c2)--(v1);
\draw (c2)--(v2);
\draw (c2)--(v3);
\draw (c2)--(v4);
\draw (c2)--(v5);
\draw (c2)--(v6);


\end{tikzpicture}

%% file: Hardness.tex
We formally state the decision problem whether a given graph $G$ admits a \qnbc.

\vspace{0.2cm}
\noindent\fbox{%
  \begin{minipage}{\linewidth}
    {\textbf{\textsc{Quasi  Neighborhood Balanced Coloring}}}\\
    \textbf{Input:} A graph $G$. \hfill\\
    \textbf{Question:} Is there a vertex coloring of $G$ using colors red and blue, such that for any vertex, the difference between the number of red and blue neighbors is at most one, and there is at least one vertex where this difference is exactly one. 
  \end{minipage}
}

\vspace{0.3cm}

Our reduction relies on the \nbc problem, whose \NP-completeness is established in \cite{knbc}.

\vspace{0.2cm}
\noindent\fbox{
  \begin{minipage}{\linewidth}
    {\textbf{\textsc{Neighborhood Balanced Coloring Problem}}} \hfill\\
    \textbf{Input:} A graph $G$.\\[3pt]
    \textbf{Question:} Is there a coloring of $G$ using colors red and blue, such that for any vertex, there are an equal number of neighbors of both colors.
  \end{minipage}
}

\vspace{0.3cm}
\begin{theorem}
    The Quasi  Neighborhood Balanced Coloring problem is \NP-complete.
\end{theorem}
   \begin{proof}
We give a reduction from the neighborhood balanced coloring  problem.  
Let $G$ be a graph that admits a \nbc. We construct from $G$ a graph $G'$ that admits a quasi  neighborhood balanced coloring.

Let $e$ be any red--blue edge of $G$ whose endpoints are $a_e$ and $b_e$.  
Introduce a new vertex $v_e$ and make it adjacent to both $a_e$ and $b_e$.  
Assign the color red to $v_e$.  
Denote the resulting graph by $G'$.

In $G'$, every vertex of $G$ other than $a_e$ and $b_e$ still has an equal number of red and blue neighbors, exactly as in the original \nbc of $G$.  
However, the vertices $a_e$ and $b_e$ each acquire one additional red neighbor, namely $v_e$.  
Thus each of them has red--blue neighbor difference exactly one, while all other vertices have difference zero.  
Hence $G'$ is a nearly neighborhood balanced colorable graph.

\medskip
Conversely, suppose that $G'$ is \qnbc.  
Observe that the only vertices of odd degree in $G'$ are $a_e$ and $b_e$, and both are adjacent solely to the added vertex $v_e$ through which this imbalance occurs.  
Since $a_e$ and $b_e$ each differ by exactly one in their counts of red and blue neighbors, and both obtain this imbalance from the same vertex $v_e$, deleting $v_e$ restores balance at both vertices.

Remove $v_e$ from $G'$ to obtain the graph $G$.  
All vertices of $G$ then have equal numbers of red and blue neighbors, establishing that $G$ admits a neighborhood balanced coloring.
\end{proof}